\documentclass[12pt]{amsart}

\pdfoutput=1

\usepackage{amssymb,amsfonts,latexsym,amsmath,amsthm,graphicx}

\numberwithin{equation}{section}
\newtheorem{thm}{Theorem}[section]
\newtheorem{prop}[thm]{Proposition}
\newtheorem{cor}[thm]{Corollary}
\newtheorem{lemma}[thm]{Lemma}

\theoremstyle{remark} 
\newtheorem{remark}[]{Remark}

\newtheorem*{definition}{Definition}

\newcommand{\RR}{\mathbb{R}}
\newcommand{\CC}{\mathbb{C}}

\usepackage{xcolor}

\newcommand{\ol}{\overline}

\newcommand{\Vol}{\text{Vol}}

\newcommand{\p}{\partial}

\newcommand{\be}{\begin{equation}}
\newcommand{\ee}{\end{equation}}

\begin{document}

\title[Ellipsoids in potential theory]{A tale of ellipsoids in potential theory}
\date{July 2013}

\author[Dmitry Khavinson]{Dmitry Khavinson }
\address{Dept. of Mathematics and Statistics, University of South Florida, Tampa, FL 33620}

\author[Erik Lundberg]{Erik Lundberg}
\address{Dept. of Mathematics, Purdue University, West Lafayette, IN 47906}

\thanks{The first author acknowledges support from the NSF grant DMS - 0855597}

\keywords{ellipsoid, mean value property, Dirichlet's problem, Ivory's theorem, Newton's theorem}


\begin{abstract}

Ellipsoids possess several beautiful properties associated with classical potential theory.
Some of them are well known, and
some have been forgotten.
In this article we hope to bring a few of the ``lost'' pieces of classical mathematics back to the limelight.

\end{abstract}

\maketitle

\section{Dirichlet's problem}

Let us start our story with
the Dirichlet problem.
This problem of finding a harmonic function
in a, say, smoothly bounded domain $\Omega \subset \RR^n$
matching a given continuous function $f$ on $\partial \Omega$
gained huge attention in the second half of the nineteenth century due to its central role in Riemann's proof
of the existence of a conformal map of any simply connected domain onto the disk.
Later on Riemann's proof was criticized by Weierstrass, and,
after a considerable turmoil,
corrected and completed by Hilbert and Fredholm - cf. \cite{Reid96}
for a very nice historical account.
Here, we want to focus on algebraic properties of solutions to the Dirichlet problem when 
$\Omega$ is an ellipsoid and the data $f$ possess nice algebraic properties.
Thus, we first present the following proposition.

\begin{prop}\label{prop:DP}
Let $$\Omega = \left\{ x \in \RR^n : \sum_{j=1}^n \frac{x_j^2}{a_j^2} -1 \leq 0, \quad a_1 \geq a_2 \geq .. \geq a_n > 0 \right\}$$
be an ellipsoid.
The solution $u$ to the Dirichlet problem
\begin{equation}
\label{eq:DP}
\left\{
\begin{array}{l}
\Delta u = 0  \quad \text{in } \Omega \\
u|_{\partial \Omega} = p
\end{array}\right.,
\end{equation}
where $p$ is a polynomial of $n$ variables, is a harmonic polynomial.
Moreover,
\begin{equation}\label{eq:deg}
\deg u \leq \deg p. 
\end{equation}
\end{prop}

\remark
Proposition \ref{prop:DP} was widely known in the nineteenth century for $n=2,3$ (perhaps due to Lam\'e)
and was proved with the use of ellipsoidal harmonics (see \cite[Section 58]{Akh90} for a discussion of Laplace's equation in elliptic coordinates).
It is still widely known nowadays for balls but often disbelieved for ellipsoids.
The elder author has won a substantial number
of bottles of cheap wine betting on its truthfulness at various math events and then producing the following proof
that was related to him by Harold S. Shapiro. 
We do not know who thought of it first, but we hope the reader will agree that it deserves to be called, following P. Erd\"os, the ``proof from the book''.

\begin{proof}
Denote by $P_{n,m} =  P_m$ the finite-dimensional space of polynomials of degree $\leq m$ in $n$ variables.
Let $q(x) = \sum{\frac{x_j^2}{a_j^2}}-1$ be the defining quadratic for $\partial \Omega$.  
Consider the linear operator $T : P_m \rightarrow P_m$ defined by
$$T ( r) := \Delta(q r).$$ 
The maximum principle yields at once that $\ker T = 0$,
so $T$ is injective.
Since $\dim P_m < \infty$, this implies that $T$ is surjective.

Hence, given $P \in P_m$ with $m \geq 2$, we can find a polynomial $r \in P_{m-2}$
such that
$T r = \Delta P$.
The function
$$u = P - q r $$
is then the solution of (\ref{eq:DP}).
\end{proof}

Proposition \ref{prop:DP} was extended \cite{KS92} to the case of entire data.
Namely, entire data $f$ (i.e., an entire function of variables $x_1, x_2, .., x_n$) yields an entire solution to the
Dirichlet problem in ellipsoids.
This result was sharpened by Armitage in \cite{Arm2004}
who showed that the solution's order and type are dominated by that of the data.

One might get bold at this point and ask does the Proposition \ref{prop:DP} extend to say rational or algebraic data,
i.e., does a smooth data function in (\ref{eq:DP}) that is a rational (algebraic) function of $x_1,x_2,..,x_n$
imply rational (algebraic) solution $u$?
The answer is a resounding ``no'' but the proofs become technically more involved - see \cite{BEKS2006, BEKS2007, Eben92, EKS2005}.

\subsection{The Dirichlet problem, ellipsoids, and Bergman orthogonal polynomials}

It was conjectured in \cite{KS92} that Prop. \ref{prop:DP} (without the degree condition (\ref{eq:deg}))
characterizes ellipsoids.
Recently, using ``real Fischer spaces'', H. Render confirmed this conjecture for many
algebraic surfaces \cite{Render}.
In two dimensions, the conjecture was confirmed under a degree-related condition on the solution in terms of the data \cite{KhSt}.
This utilized a suprising equivalence, established by N. Stylianopoulos and M. Putinar \cite{PutSty}, 
of the conjecture to the existence of finite-term recurrence relations for Bergman orthogonal polynomials.
In order to state the degree conditions and the associated recurrence conditions, assume $\Omega$ is a domain in $\RR^2$ with $C^2$-smooth boundary.
Let $\{p_m\}$ be the Bergman orthogonal polynomials (orthogonal w.r.t. area measure over $\Omega$),
and consider the following properties for $\Omega$.

\begin{enumerate}
\item  There exists $C$ such that for a polynomial data of degree $m$ there always exists a polynomial solution of the Dirichlet problem posed on $\Omega$ of degree $\leq m + C$.

\item  There exists $N$ such that for all $k,m$, the solution of the Dirichlet problem with data $\bar{z}^kz^m$ is a harmonic polynomial of degree $\leq (N-1)k+m$
in $z$ and of degree $\leq (N-1)m +k$ in $\bar{z}$.

\item  There exists $N$ such that $\{p_m\}$ satisfy a finite $(N+1)$-recurrence relation, i.e. there are constants $a_{m-j,m}$ such that
$$zp_m = a_{m+1,m}p_{m+1}+a_{m,m}p_m+...+a_{m-N+1,m}p_{m-N+1}.$$

\item  The Bergman orthogonal polynomials of $\Omega$ satisfy a finite-term recurrence relation, i.e., for every fixed $\ell > 0$, there exists an $N(\ell) > 0$, such that
$a_{\ell,m} = \langle zp_m, p_\ell \rangle = 0$, $m \geq N(\ell)$.

\item  For any polynomial data there exists a polynomial solution of the Dirichlet problem posed on $\Omega$.
\end{enumerate}

Properties $(4)$ and $(5)$ are essentially equivalent \cite{PutSty}, and $(1) \Rightarrow (2)$, $(2) \Leftrightarrow (3)$, and $(3) \Rightarrow (4)$.  
In \cite{KhSt} the authors used ratio asymptotics of orthogonal polynomials to show that $(2)$ and equivalently $(3)$ each characterize ellipses.  
The weaker statement that $(1)$ characterizes ellipsoids was proved in arbitrary dimensions \cite{LundRend2011}.
For more about the Khavinson-Shapiro conjecture stated in \cite{KS92}, we refer the reader to \cite{ChS2001, Eben92, KhLund, KhSt, KS89, Khav96, Lund2009, LundRend2011, PutSty, Render, Render2010} and the references therein.

\section{The mean value property for harmonic functions}\label{sec:MVP}

The mean value property for harmonic functions can be rephrased as saying that
\emph{the average of any harmonic function over concentric balls is a constant}.
As we formulate precisely below, 
there is a mean value property for ellipsoids which says \emph{the average of any harmonic function over confocal ellipsoids is a constant}.

Consider a heterogeneous ellipsoid
$$\Gamma:= \left\{ x\in\mathbb{R}^N:\sum_{j=1}^N\frac{x_j^2}{a_j^2}-1=0, \quad a_1>a_2>\cdots >a_N>0 \right\},$$
and let $\Omega$ be its interior.

\begin{definition}
A family of ellipsoids $\left\{\Gamma_\lambda\right\}$,
$$
\Gamma_\lambda=\left\{x\in\mathbb{R}^N:\sum_{j=1}^N\frac{x_j^2}{a_j^2+\lambda}-1=0\right\},
$$
where $-a_N^2<\lambda<+\infty$ is called a confocal family \textup{(}for $N=2$ these are ellipses with the same foci\textup{)}.
\end{definition}

Note that the shapes of confocal ellipsoids differ, and 
as $\lambda \rightarrow \infty$, $\Gamma_\lambda$ look like a spheres.

Observe that when $\lambda \to - a_N^2$,
$$
\Gamma_\lambda\to\left\{x\in\mathbb{R}^N:
x_N=0,\sum_{j=1}^{N-1}\frac{x_j^2}{a_j^2+\lambda}-1=0\right\}=:E.
$$
$E$ is called the \textit{focal ellipsoid}.

The following classical theorem goes back to MacLaurin who considered prolate spheroids in $\RR^3$ ($a_1 > a_2 =a_3$). 
General ellipsoids were treated later by Laplace \cite[Ch. 2]{Mac58}.

\begin{figure}[h]
    \includegraphics[scale=.15]{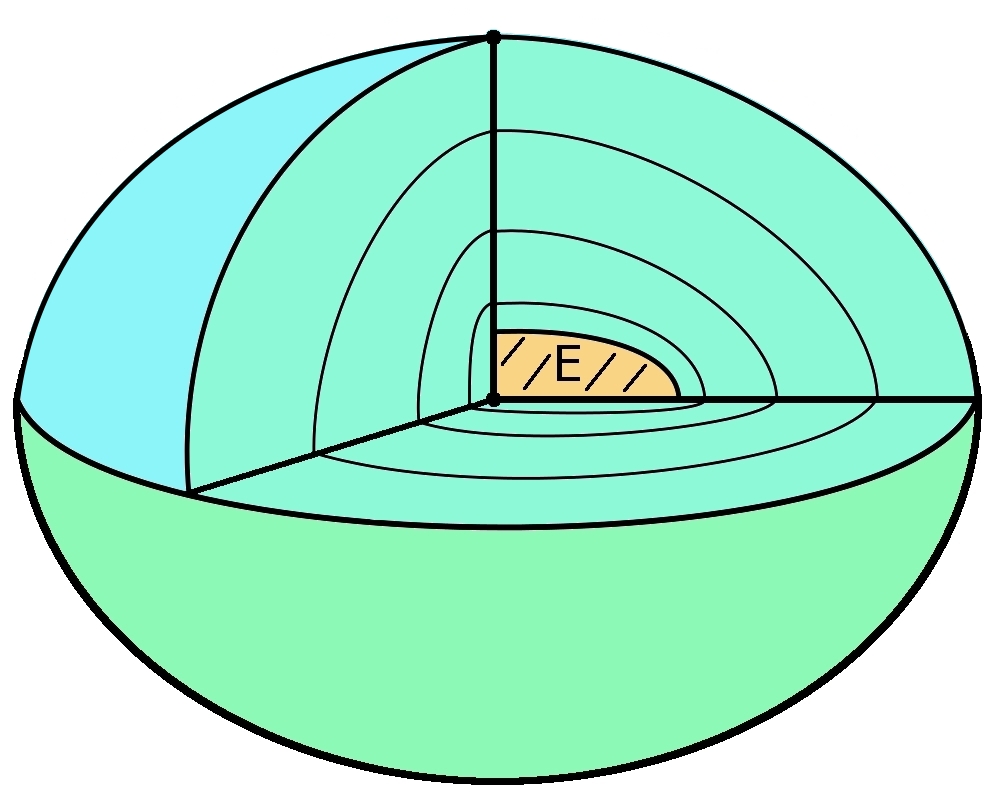}
    \caption{The mean value over confocal ellipsoids is constant.}
    \label{fig:confocal}
\end{figure}

\begin{thm}
Let $u$ be, say, an entire harmonic function. Then
\begin{equation}\label{eq:Mac}
\frac{1}{| \Omega_\lambda|}\int_{\Omega_\lambda}u(x) dx=\text{const.}
\end{equation}
for all $\lambda:\lambda>-a_N^2$.
\end{thm}

From now on, for the sake of brevity, we shall only consider the case $N\ge 3$.

\remark MacLaurin's theorem is a corollary (via a simple change of variables, see \cite[Ch. VI, Sec. 16]{CH} or \cite[Ch. 13]{Khav96}) of the following result of \'Asgeirsson \cite{Asg37}.

``Suppose $u=u(x,y)$, where $x\in\mathbb{R}^{m_1}$, $y\in\mathbb{R}^{m_2}$ satisfy the ultrahyperbolic equation
$$
\Delta_xu=\Delta_yu.
$$
Then if $\mu_i(x,y,r)$, $i=1,2$ denote respectively the mean values of $u$ over $m_i$-dimensional balls of radius $r$ centered at $(x,y)$, we have $\mu_1(x,y,r)=\mu_2(x,y,r)$.''

Here, we offer a purely algebraic approach to MacLaurin's theorem \cite{KS89}, \cite[Ch. 13]{Khav96}.
The following notions are due to E. Fischer \cite{Fischer} (see also \cite[Ch. IV]{SteinWeiss1971}). 
Let $H_k$ be the space of homogeneous polynomials of degree $k$. If $f\in H_k$, then
$$
f(z)=\sum_{|\alpha|=k}f_\alpha z^\alpha.
$$
Introduce an inner product on $H_k$ (called the Fischer inner product), by letting
\begin{equation}\label{eq:Fischer}
\left\langle z^\alpha,z^\beta\right\rangle
=\begin{cases}
0, & \alpha\ne\beta \\
\alpha!, & \alpha=\beta
\end{cases}.
\end{equation}
If $f=\sum\limits_{|\alpha|=k}f_\alpha z^\alpha$, $g=\sum\limits_{|\alpha|=k}g_\alpha z^\alpha$ then $\displaystyle \langle f,g\rangle=\sum_{|\alpha|=k}a!f_\alpha\overline{g_\alpha}. $

The main point of introducing such an inner product is that 
the operators $\left( \frac{\partial}{\partial z} \right)^\alpha$ and multiplication by $z^\alpha$ are adjoint with respect to the Fischer inner product.

Let $\mathcal{H}_m$ denote the space of homogeneous harmonic polynomials of degree $m$.
It follows from the definition (\ref{eq:Fischer}) that
$\frac{1}{m!}\left( z\cdot\bar{\xi}\right)^m $ 
is a reproducing kernel for $\mathcal{H}_m$, i.e., for all $f\in \mathcal{H}_m$,
\begin{equation*}
\frac{1}{m!}\left\langle f, \left( z\cdot\bar{\xi}\right)^m\right\rangle=f(\xi).
\end{equation*}
This fact, along with Hilbert's Nullstellensatz, easily yields the following lemma (see \cite[Ch. 13]{Khav96} for a detailed proof).

\begin{lemma}\label{lemma:span}
The linear span of harmonic polynomials $\left( z\cdot\bar{\xi} \right)^m$ for all $\xi\in\Gamma_0=\left\{\xi\in\mathbb{C}^N:\sum\limits_{j=1}^N\xi_j^2=0\right\}$ \textup{(}the isotropic cone\textup{)} equals $\mathcal{H}_m$.
\end{lemma}

\begin{proof}[Proof of MacLaurin's Theorem]
It suffices to check (\ref{eq:Mac}) for harmonic homogeneous polynomials,
and in view of Lemma \ref{lemma:span}, we just have to check it for polynomials
$$
\left( z\cdot\bar{\xi}\right)^m,\quad\xi\in\Gamma_0.
$$
Fix $\lambda$. Let $b_i=\left(a_i^2+\lambda\right)^{1/2}$ be the semi-axes of $\Omega_\lambda$. We have to show that
$$
\frac{1}{| \Omega_\lambda |}\int_{\Omega_\lambda}\left( x\cdot\bar{\xi}\right)^mdx
=\frac{1}{|\Omega|}\int_\Omega\left( y\cdot\bar{\xi}\right)^mdy,\quad
\forall\xi\in\Gamma_0.
$$
Changing variables in both integrals 
$x_k = a_k x_k'$, $y_k = b_k y_k'$
we see that it suffices to show the following:
$$
\int_B\left(\sum_{k=1}^Na_k x_k\overline{\xi_k}\right)^mdx
=\int_B\left(\sum_{k=1}^N b_k x_k\overline{\xi_k}\right)^mdx,
$$
where $B$ is the unit ball in $\mathbb{R}^N$. 
Or, since for $\xi\in\Gamma_0$
$$
\sum_{k=1}^N \left( (a_k \xi_k)^2 - (b_k \xi_k)^2 \right) = -\lambda^2 \sum_{k=1}^N \xi_k^2 = 0,
$$
it suffices to check the following assertion.

\textbf{Assertion.} The polynomial
$$
P(t):=\int_B\left(\sum_{k=1}^Nt_kx_k\right)^mdx
$$
depends only on $\sum\limits_{k=1}^Nt_k^2$, for $t\in\mathbb{C}^N$.

The assertion follows at once from the rotation invariance of $P$ \cite[Ch. 13]{Khav96}, \cite{KS89}.
\end{proof}

The following application is noteworthy.
Let $\Omega$ be an ellipsoid with semiaxes $a_1>a_2>\dotsb>a_N>0$, and
$$
u_\Omega(x):=C_N\int_\Omega\frac{dy}{|x-y|^{N-2}},\quad x\in\mathbb{R}^N\setminus\Omega
$$
be the exterior potential of $\Omega$.

Recall that $E$ denotes the focal ellipsoid defined above.  The following corollary of MacLaurin's theorem describes a so-called \emph{mother body} \cite{Gust98}, 
i.e., a measure supported inside the ellipsoid which generates the same gravitational potential as the uniform density (outside the ellipsoid) 
but is in some sense minimally supported  (in this case supported on $E$, a set of codimension one with connected complement).

\begin{cor}\label{cor:motherbody}
For $x\in\mathbb{R}^N\setminus\bar{\Omega}$
$$
u_\Omega(x)=C_N\int_E\frac{d\mu(y)}{|x-y|^{N-2}},
$$
where
$$
d\mu(y)=2\left( \prod_{j=1}^Na_j\right)\left( \prod_{j=1}^{N-1}\left( a_j^2-a_N^2 \right) \right)^{-1/2}
\left(1-\sum_{j=1}^{N-1}\frac{y_j^2}{a_j^2-a_N^2}\right)^{1/2}dy'\mid_E
$$
\textup{(}$dy'$ is Lebesgue measure on $\left\{y_N=0\right\}$\textup{)}.
\end{cor}

\begin{proof}[Sketch of proof]
Since the integrand is harmonic, we have by MacLaurin's theorem
$$
u_\Omega(x)=\frac{\prod\limits_{j=1}^Na_j}
{\prod\limits_{j=1}^N\left(a_j^2+\lambda\right)^{1/2}}
\int_{\Omega_\lambda}v(y)\,dy,
$$
where we set $v(y):\frac{C_N}{|x-y|^{N-2}}$.
After simplifying this integral using Fubini's theorem, 
the corollary is established by applying the Lebesgue dominated convergence theorem as $\lambda \rightarrow - a_n^2$ \cite[Ch. 13]{Khav96}.
\end{proof}

We note in passing that finding relevant mother bodies for oblate and prolate spheroids (supported on a disk and segment respectively) could be a satisfying exercise.

Since the density of the distribution $d\mu$ is real analytic in the interior of $E$ (viewed as a set in $\RR^{n-1}$) we note the following corollary:
\begin{cor}
The potential $u_\Omega(x)$ extends as a \textup{(}multivalued\textup{)} harmonic function into $\mathbb{R}^N\setminus\partial E$.
\end{cor}

An extension of this fact and a ``high ground'' explanation based on holomorphic PDE in $\CC^n$ is discussed in Section \ref{sec:CP}.

\section{The equilibrium potential of an ellipsoid. Ivory's Theorem}\label{sec:Ivory}

Considering that force is the gradient of potential,
the following theorem, due to Newton, can be paraphrased in a rather catchy way: ``there is no gravity in the cavity''.

\begin{thm}[Newton's theorem]
\label{thm:Newton}
Let $t>1$, and consider the ellipsoidal shell $S := t\Omega \setminus \Omega$ between two homothetic ellipsoids.
The potential $U_S$ of uniform density on $S$ is constant inside the cavity $\Omega$. 
\end{thm}

In fact, ellipsoids are characterized by this property, i.e., Newton's theorem has a converse \cite{DF86, Di, Ni, Karp, Khav96}.

A consequence of Newton's theorem is that the gravitational potential $U_\Omega$ of $\Omega$ is a quadratic polynomial inside $\Omega$. 
Namely,
$$U_\Omega(x) = B - \sum_{i=1}^N A_j x_j^2, \quad \text{for } x \in \Omega,$$
with $B=C_N \int_\Omega \frac{dV(y)}{|y|^{N-2}} = U_\Omega(0)$,
where $C_N=\frac{1}{\Vol (S^{N-1})}$.
Indeed, denoting by $\Omega_t = t \Omega$ (for $t>1$) the dilated ellipsoid, 
one computes that its gravitational potential is $u_t(x)=t^2u(x/t)$. 
Since Newton's theorem implies that ($u$ is the potential of the original ellipsoid), 
$u_t-u =const$ inside $\Omega$, the smaller ellipsoid, 
then taking partial derivatives $\partial^\alpha$, w.r.t. $x$, $|\alpha| = 2$, 
yields that $\partial^\alpha u_t(x)$= $\partial^\alpha u(x/t)=\partial^\alpha u(x)$.
Thus all these partial derivatives are homogeneous of degree zero inside $\Omega$. 
They are also obviously continuous and, hence, are constants, thus yielding $U_\Omega$ to be a quadratic as claimed.

Denoting $\Gamma := \p \Omega$, consider the single layer potential
$$ V(x) = C_N \int_{\Gamma} \frac{\rho (y)}{|x-y|^{N-2}} dA(y), $$
where
$\rho(y)$ is the mass density and $dA(y)$ on $\Gamma$ is the surface area measure.
$V(x)$ is called an \emph{equilibrium potential} if $V(x) \equiv 1$ on $\Gamma$ and hence inside $\Omega$.
For the sake of brevity we focus on $N \geq 3$ leaving the case $N=2$
as an exercise.
The quantity 
$$\sigma := \lim_{|x| \rightarrow \infty} |x|^{N-2} V(x) = C_N \int_{\Gamma} \rho(y) dA(y)$$
is called capacity.

On the way to proving Ivory's theorem, we note an explicit formula for the equilibrium potential.
Again, $N \geq 3$ \cite{Khav96, KS89}.
 

\begin{cor}
With $B$ as above, in $\RR^N \setminus \ol{\Omega}$
\begin{equation}\label{eq:Ivory}
 V(x) =  \frac{1}{B}\left( \hat{\mu} - \frac{1}{2} \sum_{i=1}^N x_i \frac{\p \hat{\mu}}{\p x_i} \right),
\end{equation}
where $\hat{\mu}(x) = C_N \int_E \frac{d\mu(y')}{|x-y|^{N-2}}$,
$y'=(y_1,y_2,..,y_{N-1},0)$,
and $d\mu(y')$ is the MacLaurin ``quadrature measure'' supported on the focal ellipsoid $E$ (cf. Cor. \ref{cor:motherbody}).
\end{cor}

\begin{proof}

The RHS of (\ref{eq:Ivory}) is harmonic in $\RR^N \setminus \ol{\Omega}$
(in fact, in $\RR^N \setminus E$) since $\hat{\mu}$ is harmonic there and $\Delta( x \cdot \nabla \hat{\mu}) = n \Delta \mu = 0$.
On $\Gamma$, by MacLaurin's theorem and Newton's theorem
\begin{equation}\label{eq:mu}
 \hat{\mu} = U_\Omega(x) = B - \sum_{i=1}^N A_j x_j^2  .
\end{equation}
Moreover since $U_\Omega(x)$ has continuous first derivatives throughout $\RR^N$, we can differentiate (\ref{eq:mu})
on $\Gamma$ and thus obtain
\begin{equation*}
 \frac{1}{B}\left( \hat{\mu} - \frac{1}{2} \sum_{i=1}^N x_i \frac{\p \hat{\mu}}{\p x_i} \right) = \frac{1}{B} \left( B - \sum_{i=1}^N A_j x_j^2 +  \frac{1}{2}\sum_{i=1}^N 2 A_j x_j^2 \right) = 1
\end{equation*}
Thus, the RHS of (\ref{eq:Ivory}) equals $V(x)$ on $\Gamma$.
Both functions are harmonic in $\RR^N \setminus \ol{\Omega}$ and vanish at infinity and the statement follows.
\end{proof}
\begin{cor}[Ivory's theorem]
The equipotential surfaces of the equilibrium potential $V(x)$ are confocal with $\Gamma$.
\end{cor}
For the proof, one simply notes that the RHS of (\ref{eq:Ivory})
changes only by a constant factor when $\Omega$ is replaced by a confocal ellipsoid
$$\Omega_\lambda := \left\{ x : \sum_{}^N \frac{x_j^2}{a_j^2 + \lambda} \leq 1, \lambda \geq 0 \right\}.$$
Namely, $B \rightarrow B_\lambda$ while $\frac{d \mu_\lambda}{d \mu} = \frac{\Vol(\Omega_\lambda)}{\Vol(\Omega)}$.

For the classical proof of Ivory's Theorem, see \cite{Mac58}, \cite[Lecture 30]{FuchsTab2007}.

\section{Ellipsoids in fluid dynamics}

Let us pause for a moment and apply these properties of ellipsoids 
to two problems in fluid dynamics.
In the first problem, involving a slowly moving interface, viscosity plays an important role.
In the second problem, viscosity is completely neglected, while vorticity plays the dominant role.

\subsection{Moving interfaces and Richardson's theorem}

Imagine a blob of incompressible viscous fluid within a porous medium surrounded by an inviscid fluid.
Suppose there is a sink at position $x_0$ in the region occupied by viscous fluid.
Averaging the Navier-Stokes equations
over pores \cite{Brinkman49} leads to Darcy's law for the fluid velocity $v$ in terms of the pressure $P$
\begin{equation}\label{eq:Darcy}
v = -\nabla P.
\end{equation}
Incompressibility implies that 
\begin{equation*}
\nabla \cdot v = -\Delta P = 0,
\end{equation*}
except at the sources/sinks.
The pressure of the inviscid fluid is assumed constant.
Neglecting surface-tension (by far, the most controversial of these assumptions \cite{Howison86, Mineev98, Tanveer}) the pressure matches at the interface,
which gives a constant (assume zero) boundary condition for $P$,
so $P$ is nothing more than the harmonic Green's function with a singularity at $x_0$.
The mathematical problem is then to track the evolution of a domain $\Omega_t$ whose boundary velocity is determined by the gradient of its own Green's function.
See \cite{EV92} for an engaging exposition of the two-dimensional case of this problem.

Given a harmonic function $u(x)$, Richardson's theorem \cite{Rich72} describes the time dependence of 
the integration of $u$ over the domain occupied by the viscous fluid.
In the language of integrable systems this represents ``infinitely many conservation laws''.

\begin{thm}[S. Richardson, 1972]
 Let $u(x)$ be a function harmonic in $\Omega_t$ for all $t$. 
 Then
 \begin{equation}\label{eq:rich}
  \frac{d}{dt} \int_{\Omega_t} u(x) dV(x) = -Q u(x_0) ,
 \end{equation}
 where $x_0$ is the position of the sink with pumping rate $Q>0$.
\end{thm}

An alternative setup places the viscous fluid in an unbounded domain with a single sink at infinity \cite{DF86};
a reformulation of Richardson's theorem implies that the potential inside the cavity of the shell regions $\Omega_t \setminus \Omega_{s>t}$ is constant.
Thus, it is a consequence of Newton's Thm. \ref{thm:Newton} that an increasing family of homothetic ellipsoids is an exact solution. 
In fact, this is the only solution starting from a bounded inviscid fluid domain that exists for all time and fills the entire space \cite{DF86} (also, cf. \cite{Karp}).

\begin{figure}[h]
    \includegraphics[scale=.3]{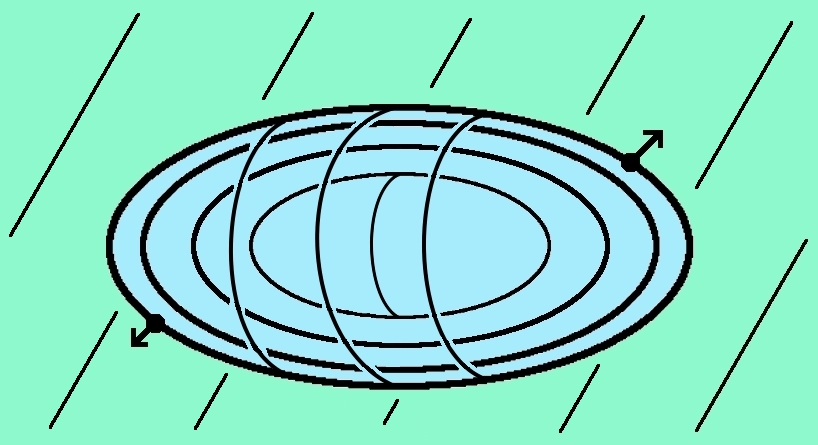}
    \caption{Viscous fluid occupies the exterior.  The ellipsoid grows homothetically.}
    \label{fig:LG}
\end{figure}

Returning to the case when the viscous fluid is bounded, suppose the initial domain $\Omega_0$ is an ellipsoid
and consider the problem of determining 
sinks and pumping rates such that $\{ \Omega_t \}_{t=0}^T$ shrinks to zero volume as $t \rightarrow T$.
As a consequence of the mean value property,
one can solve this problem exactly thus removing all of the fluid
provided we can stretch our imaginations to allow a continuum of sinks.
Starting from the given ellipsoid $\Omega_0$, 
the evolution $\Omega_t$ is a family of ellipsoids confocal to $\Omega_0$ shrinking down to the (zero-volume) focal set $E$,
and the pumping rate is given by the time-derivative of the quadrature measure appearing in Corollary \ref{cor:motherbody}.

\subsection{The quasigeostrophic ellipsoidal vortex model}

Based on the observation that motion in the atmosphere is roughly stratified into horizontal layers,
the quasigeostrophic approximation \cite{Ph63} provides a simplified version of the Euler equations (governing inviscid incompressible flow).
Further assumptions reduce the entire dynamics to
a scalar field, the potential vorticity,
which in the high Reynold's number limit, forms 
coherent regions of uniform density \cite{RY52}.
Even with these simplifications, the problem can still be quite complicated.
For instance, approximating the regions of potential vorticity by clouds of point-vortices,
one encounters the notoriously difficult $n$-body problem.

\begin{figure}[h]
    \includegraphics[scale=3]{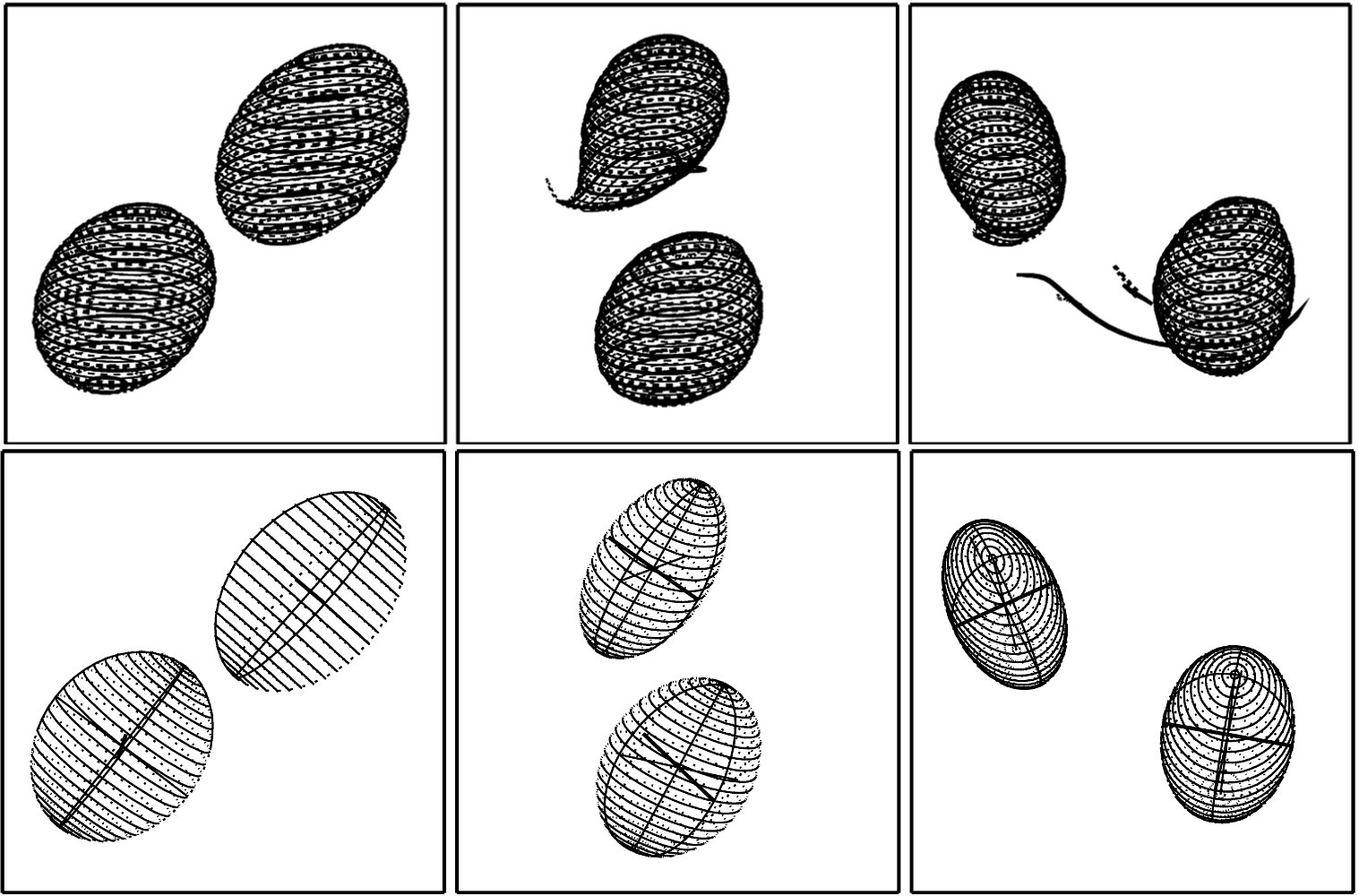}
    \caption{Top row: A vortex simulation using ``contour dynamics''.  
    Bottom row: A faster, but still accurate, simulation using the ellipsoidal vortex model.}
    \label{fig:Vortex}
\end{figure}

The \emph{quasigeostrophic ellipsoidal vortex model} developed
by Dritschel, Reinaud, and McKiver \cite{DRM2004},
simulates the interaction of ellipsoidal regions of vorticity (see Fig. \ref{fig:Vortex}, included here with their kind permission).
As these regions interact, the length and alignment of semiaxes can change, but non-ellipsoidal deformations are filtered out.
(Note that a single ellipsoid is stable for a certain range of axis ratios \cite{DSR2005}.)
The effect that one ellipsoid has on another is determined by its exterior potential,
and thus the mean value property can be used to replace the ellipsoid by a two-dimensional set of potential vorticity on its focal ellipse
(with density determined by Corollary \ref{cor:motherbody}) which can be further approximated by point vortices.

\begin{figure}[h]
    \includegraphics[scale=.2]{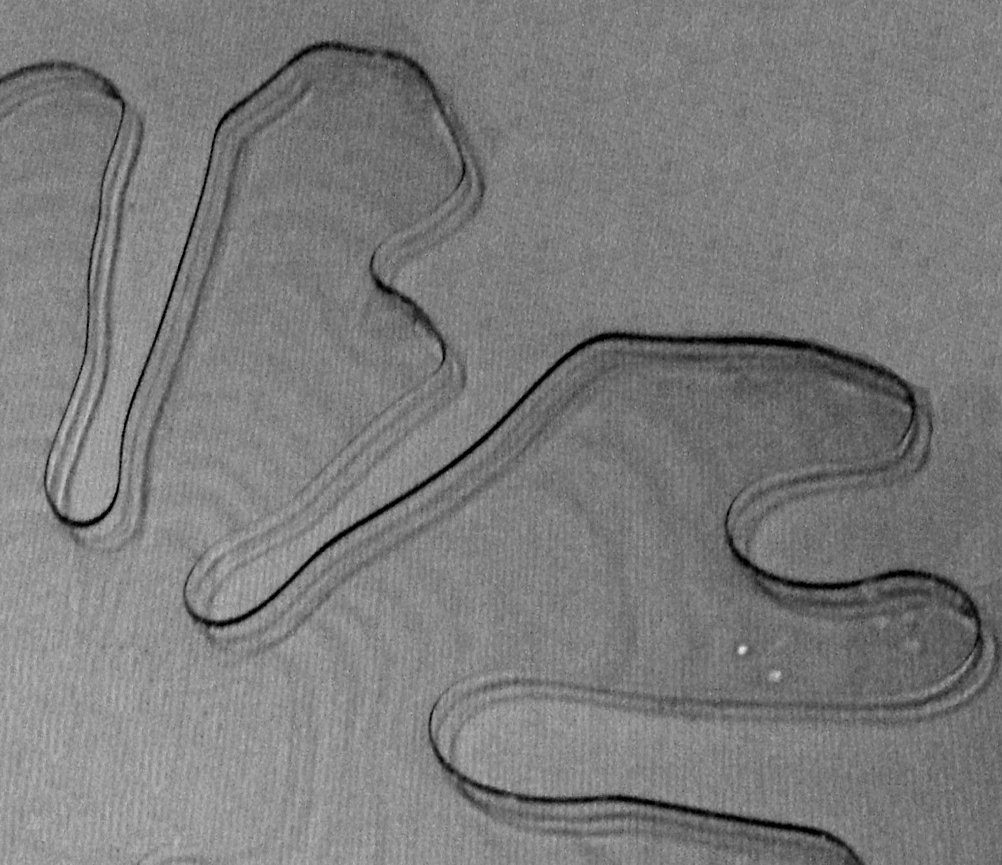}
    \caption{Viscous fingering in a Hele-Shaw cell.}
    \label{fig:HS}
\end{figure}

\remark It is interesting to single out the two-dimensional 
case of the moving interface problem which serves as a model for viscous fingering in a Hele-Shaw cell \cite{GustVas}.
Conformal mapping techniques lead to explicit exact solutions \cite{Mineev} that can even exhibit the tip-splitting depicted in Fig. \ref{fig:HS}.
The vortex dynamics problem also admits many sophisticated analytic solutions in the two-dimensional case \cite{CrowdyMarshall}. 
For a compelling survey discussing \emph{quadrature domains} as a 
common thread linking these and several other fluid dynamic problems, see \cite{Crowdy}.

\begin{figure}[h]
\begin{center}
    \includegraphics[scale=.36]{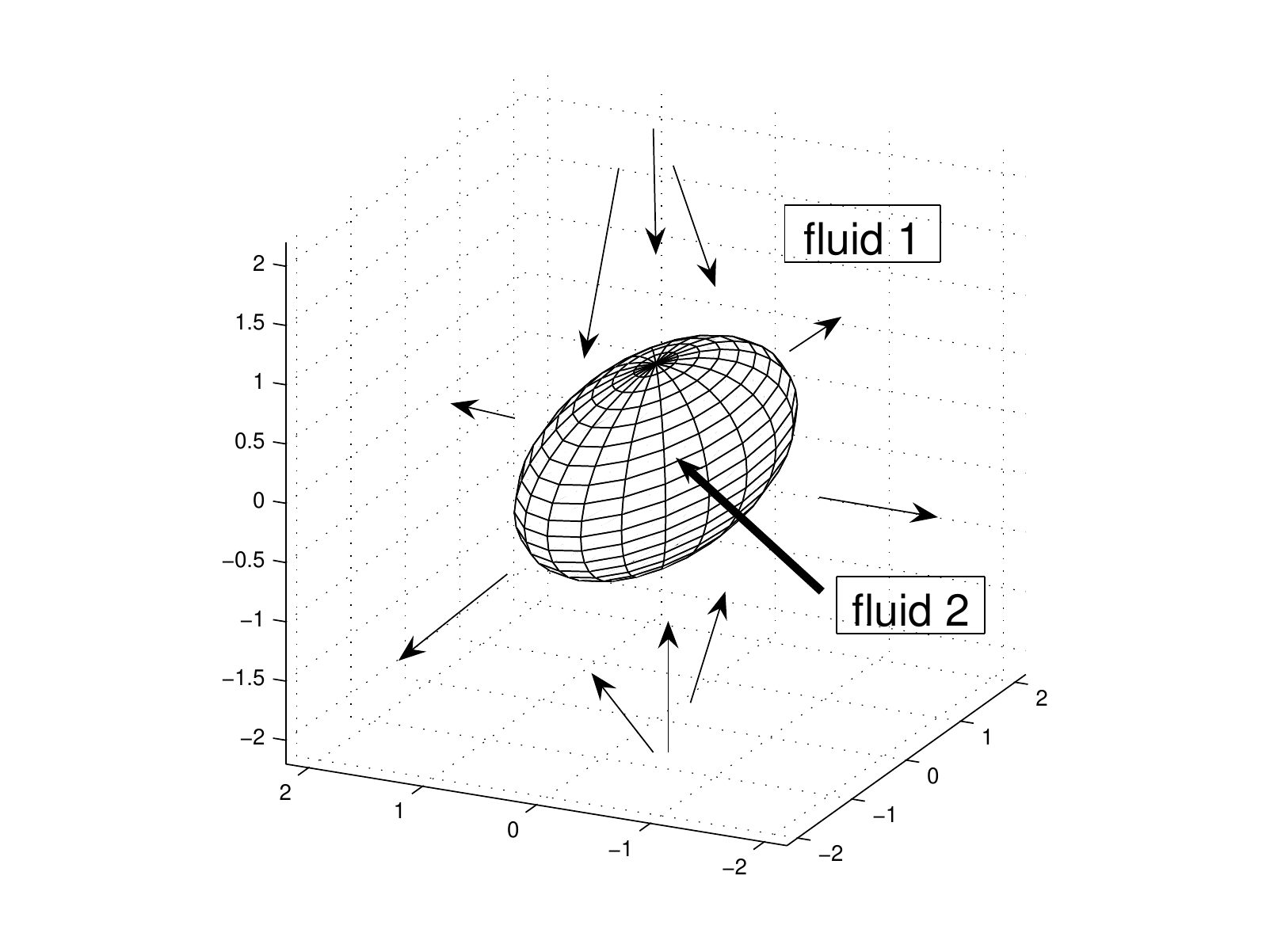}
    \includegraphics[scale=.44]{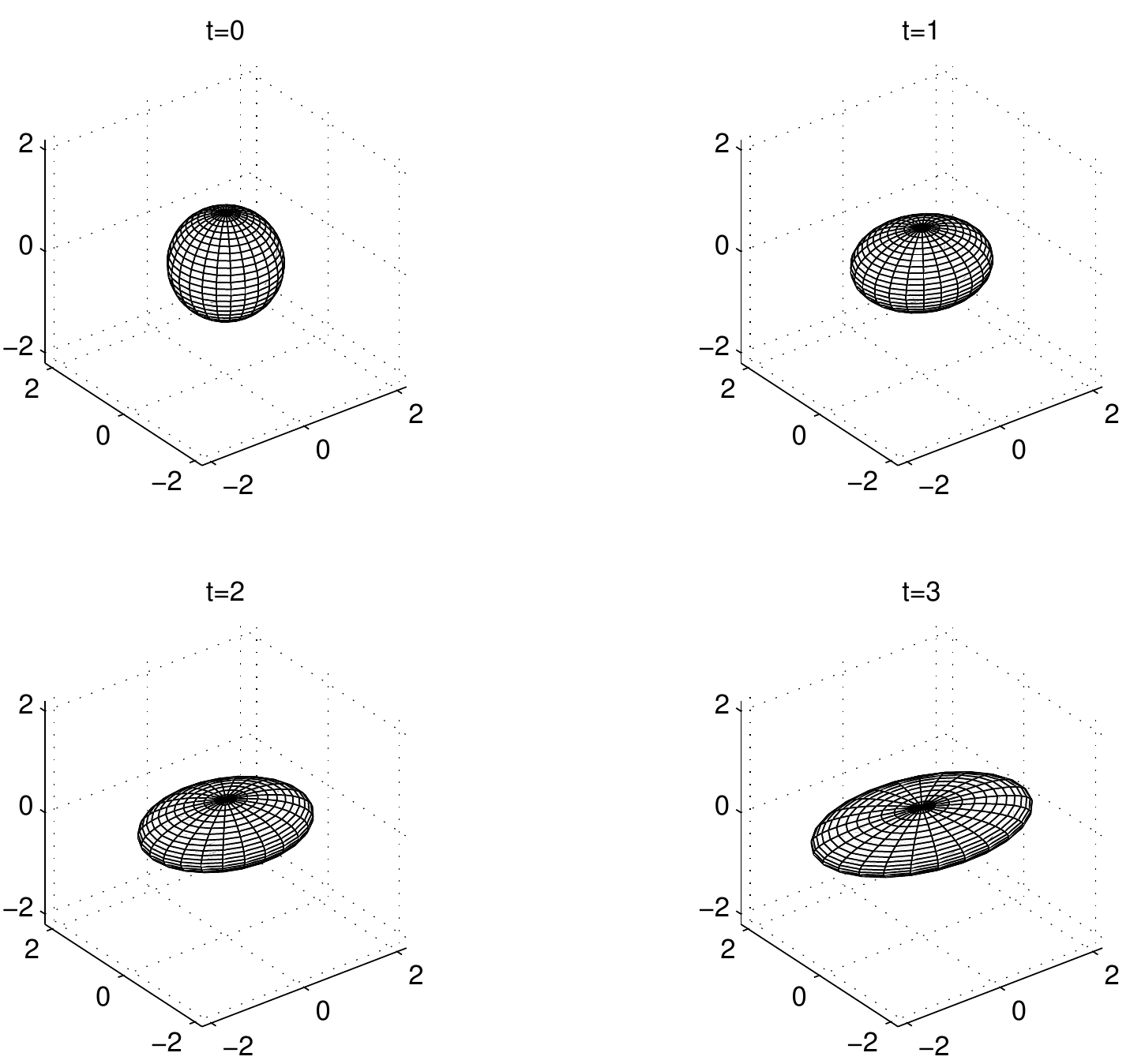}
    \caption{An ellipsoidal region of viscous fluid surrounded by a second fluid with different viscosity amid a porous medium.  
    Under a straining flow (linear in the far field), the evolution remains ellipsoidal.}
\end{center}
    \label{fig:TwoPhase}
\end{figure}

In yet another physically distinct setting, ellipsoids appear as exact solutions to a certain two-phase problem
in fluid dynamics \cite{BuchakCrowdy}.
In this case there are no sources or sinks, but rather a linear straining flow at infinity (see Fig. \ref{fig:TwoPhase}).
The (fixed-volume) ellipsoid changes shape but remains an ellipsoid (see \cite{BuchakCrowdy} for details).

\section{The Cauchy problem: A view from $\CC^n$}\label{sec:CP}

The problem mentioned in Section \ref{sec:MVP} of analytically continuing the exterior potential $U_{\Omega}$ inside the region $\Omega$ occupied by mass
was studied by Herglotz \cite{Herglotz1914},
and can be reformulated as studying the singularities of the solution to the following Cauchy problem posed on the initial surface $\Gamma:= \p \Omega$.
\begin{equation}
\label{eq:MSP}
\left\{
\begin{array}{l}
\Delta M = 1 \quad \text{near } \Gamma\\
M \equiv_{\Gamma}  0
\end{array}\right.,
\end{equation}
where the notation $M\equiv_{\Gamma} 0$ indicates that $M$ along with its gradient vanishes on $\Gamma$.

The fact that $M$ carries the same singularities in $\Omega$ as the analytic continuation $u$ of $U_\Omega$ is a consequence of the fact that
$u$ itself is given by the piecewise function
\begin{equation}\label{eq:piecewise}
u:=
\left\{
\begin{array}{l}
U_{\Omega}, \text{ outside } \Omega \\
U_{\Omega} - M, \text{ inside } \Omega
\end{array}\right..
\end{equation}
The reason is that $u$ is harmonic on both sides of $\Gamma$ and $C^1$-smooth across $\Gamma$.
An extension of Morera's theorem (attributed to S. Kovalevskaya) implies that $u$ is actually harmonic
across $\Gamma$, i.e., $\Gamma$ is a removable singularity set for $u$.
Thus, $u$ is the desired analytic continuation of $U_\Omega$ across $\Gamma$,
and the singularities of $u$ in $\Omega$ are carried by $M$.

Further reformulating the problem, note that the so-called \emph{Schwarz potential} of $\Gamma$,
$W = \frac{1}{2}|x|^2 - M$, has the same singularities as $M$ and solves a Cauchy problem for Laplace's equation:
\begin{equation}
\label{eq:SP}
\left\{
\begin{array}{l}
\Delta W = 0 \quad \text{near } \Gamma \\
W \equiv_{\Gamma} \frac{1}{2} |x|^2
\end{array}\right..
\end{equation}

This is a rather delicate (ill-posed according to Hadamard) problem,
and our discussion of it will pass from $\RR^n$ to the complex domain $\CC^n$.
Let us first consider
the more intuitive Cauchy problem for a hyperbolic equation where similar behavior can be observed while staying in the real domain.
\begin{equation}
\label{eq:hyperbolic}
\left\{
\begin{array}{l}
v_{xy} = 1 \quad \text{near } \gamma \\
v \equiv_{\gamma} 0
\end{array}\right.,
\end{equation}
where $\gamma$ is, say, a real analytic curve in $\RR^2$.

For hyperbolic equations the mantra is ``singularities propagate along characteristics''.
If the solution is singular at some point $(x_0,y_0)$, 
then one can trace the source of this singularity back to $\gamma$ by following the characteristic cone with vertex at $(x_0,y_0)$.
One expects to find a singularity in the data itself at a point where this cone intersects $\gamma$,
but what if the data function has no singularities as in (\ref{eq:hyperbolic})?
It is still possible for a singularity to propagate to the point
$(x_0,y_0)$ if the characteristic cone from $(x_0,y_0)$ is tangent to $\gamma$.
The point of tangency is called a \emph{characteristic point} of $\gamma$.

\begin{figure}[h]
    \includegraphics[scale=.5]{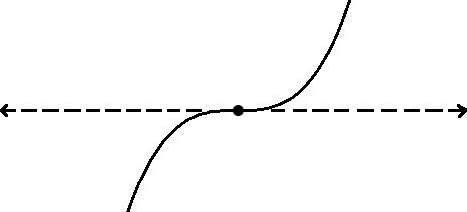}
    \caption{The solution to (\ref{eq:hyperbolic}) is regular except on the tangent characteristic $\{y=0\}$.}
    \label{fig:char}
\end{figure}

For example, suppose $\gamma := \{ y=x^3 \}$.
We can solve (\ref{eq:hyperbolic}) exactly:
$$v(x,y) = x \cdot y - \frac{x^{4}}{4} - \frac{3}{4}y^{4/3}.$$
The solution is singular on the characteristic $\{y=0 \}$ which is tangent to the initial curve $\gamma$ at the point $(0,0)$.

The singularities in the solution of (\ref{eq:SP}) also propagate along tangent characteristics,
but the characteristic points (the ``birth places'' of singularities) reside on the complexification of $\Gamma$,
the complex hypersurface given by the same defining equation.

\begin{figure}[h]
    \includegraphics[scale=.4]{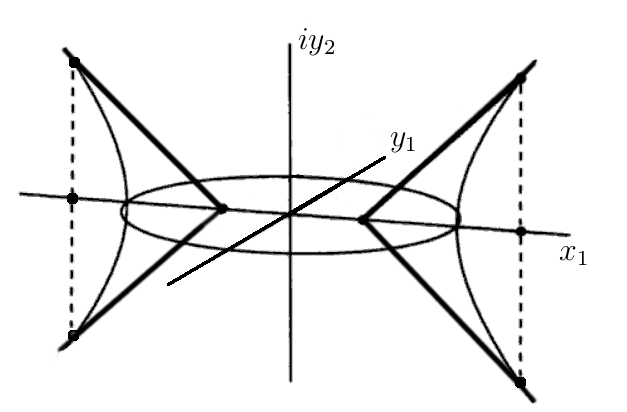}
    \caption{The characteristic lines tangent to $\Gamma$ at four characteristic points intersect $\RR^2$ precisely at the foci.}
    \label{fig:CS}
\end{figure}

\begin{thm}[G. Johnsson, \cite{Johnsson94}]
All solutions of the Cauchy problem (\ref{eq:SP})
with entire data $f$ on $\Gamma := \left\{z\in\mathbb{C}^n:\sum\limits_1^nz_1^2/a_i^2=1\right\}$ 
extend holomorphically along all paths in $\CC^n$ that avoid the characteristic surface $\Sigma$ 
(consisting of all characteristic lines tangent to $\Gamma$).
\end{thm}

The intersection $\Sigma \cap \RR^n = E$ is the focal ellipsoid.
According to the properties of the Schwarz potential discussed above, 
this provides a $\CC^n$-explanation of a rather physical fact that $E$ supports a measure solving an inverse potential problem.
Johnsson's proof of this theorem can be described as a globalization of Leray's principle, 
a local theory governing propagation of singularities.
As Johnsson notes, there is an unexpected coincidence between potential-theoretic foci (points where singularities of $W$ are located)
and algebraic foci in the classical sense of Pl\"ucker \cite{Johnsson94}.
Understanding this correspondence and extending it to higher-degree algebraic surfaces
is part of a program advocated by the first author and H. S. Shapiro.
The case $n=2$ is more transparent \cite{KS90},
and for $n>2$ it is virtually unexplored except for some axially-symmetric fourth-degree examples \cite{Lund2011}.

\section{Epilogue}

Newton's theorem can be reformulated in terms of a single layer potential obtained by shrinking a constant-density ellipsoidal shell to zero thickness (while rescaling the constant),
leading to a non-constant density $\rho(x) = 1 / |\nabla q(x)|$, where $q(x)$ is the defining quadratic of the ellipsoid.
This is sometimes called the \emph{standard single layer potential} (it is different from the equilibrium potential discussed in Section \ref{sec:Ivory}).
The modern approach due to V. I. Arnold and, then, A. Givental \cite{Arnold, Givental}, views
the force at $x_0$ induced by infinitesimal charges at two points $x_1, x_2$ on a line $\ell$ through $x_0$ as a sum of residues for a contour integral in the complex extension $L$ of $\ell$.
The vanishing of force follows from deforming the contour to infinity.
The detailed proof can be found in \cite[Ch. 14]{Khav96}.

\begin{figure}[h]
    \includegraphics[scale=.3]{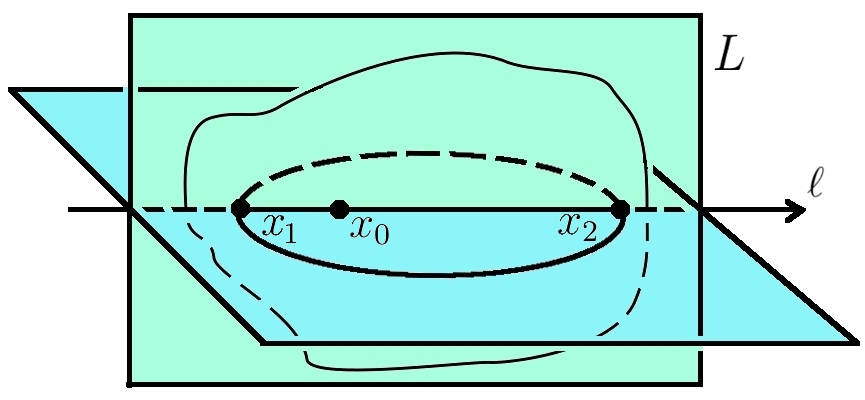}
    \caption{The force from two points is realized as a sum of residues in the complex line $L$. }
    \label{fig:givental}
\end{figure}

The same proof can be used to extend Newton's theorem beyond ellipsoids
to any \emph{domain of hyperbolicity} of a smooth, irreducible real algebraic variety $\Gamma$ of degree $k$.
A domain $\Omega$ is called a domain of hyperbolicity for $\Gamma$ if for any $x_0 \in \Omega$,
each line $\ell$ passing through $x_0$ intersects $\Gamma$ at precisely $k$ points.
For example, the interior of an ellipsoid is a domain of hyperbolicity,
and if a hypersurface of degree $2k$ consists of an increasing family of $k$ ovaloids then the smallest one is the domain of hyperbolicity.

Defining the standard single layer density on $\Gamma$ in exactly the same way as before, except that
the sign $+$ or $-$ is assigned on each connected component of $\Gamma$ depending whether the number of obstructions for ``viewing'' this component from the domain of hyperbolicity of $\Gamma$ is even or odd,
the Arnold-Givental generalization of Newton's theorem implies, in particular, that the force due to the standard layer density vanishes inside the domain of hyperbolicity (cf. \cite{Arnold, Givental} for more general statements and proofs).

As a final remark, returning to ellipsoids, and even taking $n=2$,
let us note an application to gravitational lensing of Corollary \ref{cor:motherbody}.
The two-dimensional version of MacLaurin's theorem
plays a key role in formulating analytic descriptions for the
gravitational lensing effect for certain elliptically symmetric lensing galaxies \cite{FKK, KhLund, BE, Rhie} (cf. \cite{KhN, Petters, PLW} for terminology).
Here the projected mass density that is constant on confocal ellipses produces at most 4 lensed images \cite{FKK}. 
The density that is constant on homothetic ellipses produces at most 6 images \cite{BE}, also cf. \cite{KhLund}. 
The same technique that applies MacLaurin's theorem to density that is not constant but is constant on each scaled ellipse
can also be applied to the case when the ellipses are allowed to rotate as they are scaled.
This leads to a lensing equation involving a Gauss hypergeometric function that describes the images lensed by a spiral galaxy \cite{BEFKL} (an investigation initiated during an REU).
In connection to the converse to Newton's theorem, 
whenever the rare focusing effect in graviataional lensing produces a continuous ``halo'' (aka Einstein ring - cf. \cite{KhN} for some striking NASA pictures) around the lensing galaxy (of any shape), 
the "halo" necessarily turns out to be either a circle or an ellipse \cite{FKK}. But this alley leads to the beginning of another story.

\bibliographystyle{amsplain}

\end{document}